\documentclass{amsart}[12pt]
\usepackage{mathrsfs, amsmath,amssymb,color}
\newtheorem{teo}{Theorem}
\newtheorem{pro}{Proposition}
\newtheorem{lem}{Lemma}
\newtheorem{cor}{Corollary}
\newtheorem*{rem}{Remark}
\newtheorem*{rems}{Remarks}
\title{General position of a projection and of its image under a free unitary Brownian motion }
\author[N. Demni]{Nizar Demni}
\address{IRMAR, Universit\'e de Rennes 1\\ Campus de
Beaulieu\\ 35042 Rennes cedex\\ France}
\email{nizar.demni@univ-rennes1.fr}

\author[T. Hmidi]{Taoufik Hmidi}
\address{IRMAR, Universit\'e de Rennes 1\\ Campus de
Beaulieu\\ 35042 Rennes cedex\\ France}
\email{thmidi@univ-rennes1.fr}

\begin{document}
\maketitle

\begin{abstract}
Given an orthogonal projection $P$ and a free unitary Brownian motion $Y = (Y_t)_{t \geq 0}$ in a $W^{\star}$-non commutative probability space such that $Y$ and $P$ are $\star$-free in Voiculescu's sense, the main result of this paper states  that $P$ and $Y_tPY_t^{\star}$ are in general position at any time $t$. To this end, we study the dynamics of the unitary operator $SY_tSY_t^{\star}$ where $S = 2P-1$. More precisely, we derive a partial differential equation for the Herglotz transform of its spectral distribution, say $\mu_t$. Then, we provide a flow on the interval $[-1,1]$ in such a way that the Herglotz transform of $\mu_t$ composed with this flow is governed by both the Herglotz transforms of the initial ($t=0$) and the stationary ($t = \infty)$ distributions. This fact allows to compute the weight that $\mu_t$ assigns to $z=1$ leading to the main result. As a by-product, the weight that the spectral distribution of the free Jacobi process assigns to $x=1$ follows after a normalization. In the last part of the paper, we use combinatorics of non crossing partitions in order to analyze the term corresponding to the exponential decay $e^{-nt}$ in the expansion of the $n$-th moment of $SY_tSY_t^{\star}$.  
\end{abstract}

\section{Reminder}
Let $(\mathscr{A}, \tau)$ be a $W^{\star}$-non commutative probability space with unit ${\bf 1}$ and adjoint operation $\star$: $\mathscr{A}$ is a von Neumann algebra endowed with a faithful tracial state $\tau$. In a recent paper, we studied the free Jacobi process (\cite{DHH}): this is a family of positive operators $J = (J_t)_{t \geq 0}$ valued in the compressed non commutative probability space 
\begin{equation*}
\left(P\mathscr{A}P, \tau_{P}:= \frac{1}{\tau(P)}\tau\right)
\end{equation*}
where $P \in \mathscr{A}$ is an orthogonal projection. Actually, the operator $J_t$ is defined by 
\begin{equation*}
J_t = PY_tQY^{\star}_tP
\end{equation*}
 where $Y = (Y_t)_{t \geq 0} \in \mathscr{A}$ is a free unitary Brownian motion, $Q \in \mathscr{A}$ is another orthogonal projection such that $\{P,Q\}$ and $\{Y, Y^{\star}\}$ are free families in Voiculescu's sense (see \cite{BenL}, \cite{Biane} for details). When $P=Q$ and $\tau(P) = 1/2$, the main result proved in \cite{DHH} asserts that the spectral distribution of $J_t$ in $(P\mathscr{A}P, \tau_{P})$ fits that of 
\begin{equation*}
\frac{Y_{2t} + Y_{2t}^{\star} + 2{\bf 1}}{4} 
\end{equation*}
in $(\mathscr{A}, \tau)$. Two proofs leading to this description were written in \cite{DHH}. One of them is based on the following expansion: let $S := 2P-1, \tau(S) = 0, b = Y_tSY_t^{\star}$ then 
\begin{align*}
\tau[(PY_tPY_t^{\star})^n] = \frac{1}{2^{2n+1}}\binom{2n}{n}  +  \frac{1}{2^{2n}}\sum_{k=1}^n \binom{2n}{n-k}\tau((SY_tSY_t^{\star})^k).
\end{align*}
The description of the spectral distribution of $J_t$ then follows from the fact that in this case, $SY_tSY_t^{\star}$ and $Y_{2t}$ share the same spectral distribution (\cite{DHH}, Lemma 1). For general ranks $\tau(P) = \theta \in [0,1]$, we already noticed in \cite{DHH} that 
\begin{align}\label{Limit}
\tau[(PY_tPY_t^{\star})^n] = \frac{1}{2^{2n+1}}\binom{2n}{n}  + \frac{2\theta-1}{2} + \frac{1}{2^{2n}}\sum_{k=1}^n \binom{2n}{n-k}\tau((SY_tSY_t^{\star})^k).
\end{align}
However, we inferred there that a description of the spectral distribution of $SY_tSY_t^{\star}$ for arbitrary ranks $\theta \in (0,1)$, say $\mu_t^{\theta}$, similar to that of $\mu_t^{1/2}$ seems to be quite difficult. In this paper, we shall be mostly interested in the weight that $\mu_t^{\theta}$ assigns to $z=1$ and this is two-fold. On the one hand, this weight is exactly the one that the spectral distribution of $J_t$ assigns to $x=1$. This follows from the fact we prove below that the limit as $n \rightarrow \infty$ of the RHS of \eqref{Limit} depends only on this weight, together with Lebesgue convergence Theorem applied to the LHS of \eqref{Limit} after normalizing by $1/\tau(P)$. On the other hand, a result due to J. von Neumann (\cite{vN}, Theorem 13.7. p.55) shows that for any orthogonal projections $P_1, P_2 \in \mathscr{A}$ with subspaces $H_1$ and $H_2$ respectively,  
\begin{equation*}
\lim_{n \rightarrow \infty} \tau((P_1P_2)^n) = \tau(P_1 \wedge P_2),
\end{equation*}
where $P_1 \wedge P_2$ denotes the orthogonal projection onto $H_1 \cap H_2$. Accordingly, the weight of $\mu_t^{\theta}$ at $z=1$ allows to determine whether or not the projections $P$ and $Y_tPY^{\star}$ are in general position at any time $t > 0$, that is\footnote{If two orthogonal projections are free then they satisfy \eqref{GP} (\cite{VBA}, Lemma 2.1) and are in general position according to \cite{CK}.}: 
\begin{equation}\label{GP}
\tau(P\wedge (Y_tPY_t^{\star})) \overset{?}{=} \textrm{max}(\tau(P) + \tau(Y_tPY_t^{\star}) - 1, 0) =   \textrm{max}(2\tau(P) - 1, 0).
\end{equation}
In the case of two projections $P,Q$ with equal rank $\tau(P)=\tau(Q) = 1/2$, it was recently proved in \cite{CK} that $P$ and $Y_tQY_t^{\star}$ are in general position at any time $t$: 
\begin{equation*}
\tau(P \wedge (Y_tQY_t^{\star})) = 0.
\end{equation*}
This result was then used in order to get an insight into the so-called unification conjecture for the projections $P,Q$. In the same spirit, we shall determine the weight of $\mu_t^{\theta}$ assigned to $z=1$ and prove that \eqref{GP} holds for any $\theta \in (0,1]$ and any time $t > 0$. To this end, we use free stochastic calculus to derive a recursive time-dependent relation for the moments 
\begin{equation*}
r_n^{\theta}(t) := \tau[(SY_tSY_t^{\star})^n] = \int_{\mathbb{T}}z^n \mu_t^{\theta}(dz),
\end{equation*}
$\mathbb{T}$ being the unit circle. The obtained relation is then transformed into a partial differential equation  (hereafter pde) for the Herglotz transform of $\mu_t$:
\begin{equation*}
H^{\theta}(t,z) := 1+2\sum_{n \geq 1}r_n^{\theta}(t)z^n, \quad |z| < 1.
\end{equation*}
Compared to $H^{1/2}$, the pde satisfied by $H^{\theta}$ for general $\theta$ comes with the perturbation 
\begin{equation*}
2[\tau(a)]^2\frac{z(1+z)}{(1-z)^3} = 2(2\theta -1)^2\frac{z(1+z)}{(1-z)^3} ,
\end{equation*}
while keeping the same initial data $H^{\theta}(0,\cdot) = H^{1/2}(0,\cdot)$. Using the method of characteristics, we shall find a flow $(t,z) \mapsto \psi^{\theta}(t,z)$ on $\mathbb{R}_+ \times [-1,1]$ such that 
\begin{align}\label{Dyn}
[H_{\infty}^{\theta}(\psi^{\theta}(t,z))]^2 - [H_{\infty}^{\theta}(z)]^2 = [H^{\theta}(t,\psi^{\theta}(t,z))]^2 - [H^{\theta}(0,z)]^2.
\end{align}   
Here $H_{\infty}^{\theta}$ is the Herglotz transform of the spectral distribution $\mu_{\infty}^{\theta}$ of $SUSU^{\star}$ where $U \in \mathscr{A}$ is Haar distributed (\cite{NS}). Equivalently, $\mu_{\infty}^{\theta}$ is the weak limit as $t \rightarrow \infty$ of $\mu_t^{\infty}$ and is a deformation of the Haar distribution $\mu_{\infty}^{1/2}$on $\mathbb{T}$ (since $SY_tSY_t^{\star}$ and $Y_{2t}$ are equally distributed when $\theta = 1/2$). Up to our best knowledge, no description of $\mu_{\infty}^{\theta}$ have showed up yet in literature and it is sufficient for our purposes to focus only on its weight $\mu_{\infty}^{\theta}$ at $z=1$ (which is the residue of $H_{\infty}^{\theta}$ at $z=1$). Nonetheless, we shall supply here a full description of $\mu_{\infty}^{\theta}$ relying on an explicit expression of $H_{\infty}^{\theta}$. However rather then using the analytic machinery for the multiplicative convolution of probability distributions on $\mathbb{T}$ (\cite{DNV}), we found more convenient to us to write down $H_{\infty}^{\theta}$ by taking the limit in \eqref{Limit} as $t \rightarrow \infty$ and by using the moment generating function of $PUPU^{\star}P$ in $(P\mathscr{A}P, \tau_P)$ (\cite{Demni1}). Standard analysis arguments show that $\mu_{\infty}^{\theta}$ admits an absolutely continuous part whose support consists of two symmetric (with respect to the real axis) arcs that join each other at $z = \pm 1$ if and only if $\tau(S) = 0$. As to its discrete part, it consists of the single point $z=1$ with weight $|\tau(S)|$ to $z=1$. Coming back to the flow, we shall prove that for any fixed $t > 0$ we can find a real number $z_t^{\theta} \in (0,1)$ such that $\psi^{\theta}(t,z_t^{\theta}) = 1$. Since the pole of $H^{\theta}(t,\cdot)$ at $z=1$ is exactly the weight of $\mu_t^{\theta}$ at $z=1$ and since this weight can be recovered using a radial limit along the real interval $(0,1]$, then it is given again by $|\tau(S)|$ according to \eqref{Dyn}. As a result, taking the limit as $n \rightarrow \infty$ in \eqref{Limit} yields the main result of the paper: 
\begin{teo}\label{THM}
For any $t > 0$ and any rank $\tau(P) = \theta \in (0,1)$, the orthogonal projections $P$ and $Y_tPY_t^{\star}$ are in general position.
\end{teo}
It is worth noting that this is a kind of inverse problem since we rely on the knowledge of $H_{\infty}^{\theta}$ in order to guess the behavior of $H^{\theta}(t,z)$. Conversely, our result shows that the weight $|\tau(S)|$ comes originally from the spectral distribution of $SY_tSY_t^{\star}$ and remain mysteriously unchanged up to infinity. This propagation phenomenon happens for $J_t$ as well and is a by-product of our main result: (see \cite{Demni}, p.130 for a description of its stationary distribution):
\begin{cor}\label{COR}
Both the spectral distributions of $PY_tPY_t^{\star}P$ and of $PUPU^{\star}P$ in $(P\mathscr{A}P, (1/\theta)\tau)$ assign the same weight at $x=1$: 
\begin{equation*}
\frac{1}{\theta}\max[(2\theta-1),0]. 
\end{equation*}
\end{cor}

The paper is organized as follows. We first supply a full description of $\mu_{\infty}^{\theta}$ and derive a closed formula for its moments through Jacobi polynomials. The time-dependent recursive relation for the moments $r_n^{\theta}(t), n \geq 1$ of $\mu_t^{\theta}$ comes next and is an instance of a general formula derived in \cite{BenL}. The obtained relation is then transformed it into a pde satisfied by $H^{\theta}$ whose dynamics are analyzed using the method of characteristics, leading to the flow $\psi^{\theta}$. Once we do, we prove that the limit as $n \rightarrow \infty$ of the RHS of \eqref{Limit} depends only on the weight of $\mu_t^{\theta}$ at $z=1$ and the existence of $x_t^{\theta}$ at any time $t > 0$. Doing so allows to compute this weight, proving thereby our main result and its corollary. 
We close the paper by giving a special interest in the term corresponding to the exponential decay $e^{-nt}$ in $r_n^{\theta}(t)$. When $\theta =1/2$, it is given by a Laguerre polynomial (\cite{Biane}, \cite{DHH}). For general $\theta \in (0,1)$, we shall see that the difference occurs in the value this term takes at $t=0$. Using combinatorics of non crossing partitions, we show that this value is the $n$-th even moment of the $1/2$-fold convolution of the self-adjoint operator $a_1-a_2$, where $a_1, a_2 \in \mathscr{A}$ are two free copies of $S$. Of course, the resulting convolution is not necessarily a probability measure for general $\theta \in (0,1]$ while it is obviously so when $\theta =1/2$ (it reduces the spectral distribution of $a$ since $a_1$ and $-a_2$ are equally distributed). Using the $R$-transform machinery (\cite{DNV}), we can see that the Cauchy-Stieltjes transform of the $1/2$-convolution of $a_1-a_2$ is a root of a third degree polynomial that one can express using Gauss hypergeometric functions. \\
Henceforth, we shall omit the dependence of our notations on $\theta$ for sake of clarity.      

\section{The Stationary distribution $\mu_{\infty}$} 
This section is devoted to  the Lebesgue decomposition of the  spectral {distribution $\mu_\infty$} of $SUSU^\star$, where we recall that $U \in \mathscr{A}$ is Haar unitary operator. More precisely, we show that $\mu_{\infty}$ splits into an absolutely continuous part and a singular discrete one supported in $\{1\}$ with weight $|\tau(S)|$. To proceed, we shall write down its Herglotz transform $H_{\infty}$: 
\begin{equation*}
H_{\infty}(z)= \int_{\mathbb{T}}\frac{w+z}{w-z}\mu_{\infty}(dz)  = 1 + 2\sum_{n\geq 1}r_n z^n 
\end{equation*}
where we set  
\begin{equation*}
r_n := r_n(\infty) = \tau((SUSU^{\star})^n), \quad n \geq 1.
\end{equation*}
This may be done using the free multiplicative convolution of the unitary operators $S$ and $USU^{\star}$ (\cite{NS}) whose common spectral distribution is given by 
\begin{equation*}
\theta\delta_1 + (1-\theta) \delta_{-1}.
\end{equation*}
However we found more convenient to us to deduce $H_{\infty}$ from \eqref{Limit} and from the knowledge of the moment generating function of $PUPU^{\star}P$ in $P\mathscr{A}P$ (\cite{Demni1}). The issue of our computations is  
\begin{lem}
Set $\kappa \triangleq 2\theta-1 = \tau(S)$, then  
\begin{align*}
H_{\infty}(z) = \sqrt{1+ 4\kappa^2 \frac{z}{(1-z)^2}}
\end{align*}
in some neighborhood of the origin. The equality extends analytically to the open unit disc.  
\end{lem}

\begin{proof}
Define
\begin{equation*}
m_n := \frac{1}{\tau(P)} \tau[(PUPU^{\star})^n], \, n \geq 1, \quad m_0 = 1. 
\end{equation*}
These are the moments of the stationary free Jacobi process associated with the parameters $\lambda = 1, \theta \in (0,1]$ (see \cite{Demni1} for notations, see also \cite{Collins}). For instance, equation (1) p.108 in \cite{Demni1} leads on the one side to   
\begin{align}\label{Prem}
\sum_{n\geq 1}m_nz^n = \frac{(2\theta-1) + \sqrt{1-4\theta(1-\theta)z}}{2\theta(1-z)} - 1, \, \, |z| < 1.
\end{align}
On the other side, we get after summing \eqref{Limit} over $n \geq 1$
\begin{eqnarray}\label{Deux}
\sum_{n\geq 1}m_nz^n &=& \frac{1}{2\theta}\left[\frac{1}{\sqrt{1-z}} -1 + \frac{(2\theta-1)z}{1-z}\right] +\frac{1}{\theta}\sum_{n\geq1}\frac{z^n}{2^{2n}}\sum_{k=1}^n \binom{2n}{n-k}r_k \nonumber \\
&=& \frac{1}{2\theta}\left[\frac{1}{\sqrt{1-z}} -1 + \frac{(2\theta-1)z}{1-z}\right] +\frac{1}{\theta} \sum_{k\geq1}r_k\frac{z^k}{2^{2k}}\sum_{n\geq0}^n \binom{2n+2k}{n}\frac{z^n}{2^{2n}}.
\end{eqnarray}
Using the following identity whose proof is written in \cite{DHH}:
$$
\sum_{n\geq0}^n \binom{2n+2k}{n}\frac{z^n}{2^{2n}}=\frac{2^{2k}}{\sqrt{1-z}}\left(1+\sqrt{1-z}\right)^{-2k}, \quad |z|<1, 
$$
and comparing \eqref{Prem} and \eqref{Deux}, we get
\begin{equation*}
2 \sum_{n \geq 1}r_n[\alpha(z)]^{n} = \frac{\sqrt{1-4\theta(1-\theta)z}}{\sqrt{1-z}}  - 1,\quad\hbox{with}\quad \alpha(z)=\frac{z}{(1+\sqrt{1-z})^2}\cdot
\end{equation*}
Finally, recall from \cite{DHH} that $\alpha$ is invertible from the open unit disc onto some neighborhood of the origin, where the inverse is given by 
\begin{equation*}
\alpha^{-1}(z) = \frac{4z}{(1+z)^2}.
\end{equation*}
As a result
\begin{align*}
H_{\infty}(z) = 1+ 2\sum_{n \geq 1}r_nz^n &= \frac{\sqrt{1-4\theta(1-\theta)\alpha^{-1}(z)}}{\sqrt{1-\alpha^{-1}(z)}}. 
\\& =  \sqrt{\frac{1+z^2 +2z(1-8\theta(1-\theta))}{(1-z)^2}}
\\& = \sqrt{1+ 4\kappa^2 \frac{z}{(1-z)^2}}
\end{align*}
as desired. But since  
\begin{equation*}
z \mapsto 1+ 4\kappa^2 \frac{z}{(1-z)^2} 
\end{equation*}
does not take negative values when $z$ belongs to the open unit disc, then the last statement of the lemma is clear.    
\end{proof}

\begin{cor}\label{stationar1}
The Lebesgue decomposition of the  spectral measure $\mu_\infty$ of $SUSU^\star$ is given by
$$
\mu_\infty=|\kappa|\delta_{1}+\sqrt{1-\frac{\kappa^2}{\sin^2\psi}}{\bf 1}_{\{|\sin\psi| \geq |\kappa|\}}\, d\psi.
$$
\end{cor}
\begin{proof}
From the previous lemma, $H_{\infty}$ admits a pole at $z=1$, therefore $\mu_{\infty}$  assigns a weight at $z=1$ given by 
\begin{equation*}
\frac{1}{2}\lim_{z \rightarrow 1}\sqrt{(1-z)^2+ 4\kappa^2 z} = |\kappa|.
\end{equation*}
As to the remaining parts of $\mu_{\infty}$, we first discard the values $\kappa = 0, 1$. Indeed, we know that $\mu_{\infty}$ reduces to the Haar distribution on $\mathbb{T}$ when $\kappa = 0 \Leftrightarrow \theta = 1/2$, while when $\kappa = 1 \Leftrightarrow \theta =1$ one has 
\begin{equation*}
H_{\infty}(z) = \frac{1+z}{1-z}
\end{equation*}
so that $\mu_\infty = \delta_1$. Hence, assume $|\kappa| \in (0,1)$. Then $\mu_{\infty}- |\kappa|\delta_1$ is absolutely continuous with respect to the Haar distribution in $\mathbb{T}$ and its density is given by: 
\begin{equation*}
\Re\left(\sqrt{1+ 4\kappa^2 \frac{z}{(1-z)^2}}\right), \,z \in \mathbb{T}. 
\end{equation*}
Indeed 
\begin{equation*}
H_{\infty}(z)  - |\kappa|\frac{1+z}{1-z} = \frac{(1-\kappa^2)(1-z)}{ \sqrt{z^2 + 2(2\kappa^2-1)z +1}+|\kappa|(1+z)}
\end{equation*}
has a continuous extension on the boundary $\mathbb{T}$, since 
\begin{equation*}
z \mapsto z^2 + 2(2\kappa^2-1)z +1
\end{equation*}
does not take negative values (its root lie on $\mathbb{T}$) and since the denominator does not vanish on the closed unit disc. The proposition now follows from the Poisson representation of analytic functions in the open unit disc extending continuously to $\mathbb{T}$. Finally
\begin{equation*}
\Re\left[H_{\infty}(e^{i\psi})  - |\kappa|\frac{1+e^{i\psi}}{1-e^{i\psi}}\right] = \Re\left[\sqrt{1+ 4\kappa^2 \frac{e^{i\psi}}{(1-e^{i\psi})^2}}\right] = \sqrt{1-\frac{\kappa^2}{\sin^2\psi}}{\bf 1}_{\{|\sin\psi| > |\kappa|\}}.
\end{equation*}
\end{proof}
We close this section with the following closed form of the moments $r_n, n \geq 1$ showing that these are somehow averages over $(0,|\kappa|)$ of special polynomials : 
\begin{pro}
For any $\kappa \in (-1,1]$ 
\begin{equation*}
r_n = \kappa \int_0^{\kappa} P_{n-1}^{1,0}(1-2s^2) ds. 
\end{equation*}
where $P_n^{1,0}$ is the $n$-th Jacobi polynomial of parameters $(1,0)$ (\cite{Rainville}, p.254). 
\end{pro}
\begin{proof}
Using the generalized binomial Theorem (\cite{Rainville}, p.47), we write
\begin{align*}
 \sqrt{1+ \frac{4\kappa^2z}{(1-z)^2}} -1 & = \sum_{k \geq 1} \frac{(-1/2)_k}{k!}\left[-\frac{4\kappa^2z}{(1-z)^2}\right]^k = \sum_{k \geq 1} \frac{(-1/2)_k}{k!}(-4\kappa^2z)^k \sum_{n \geq 0} \frac{(2k)_n}{n!} z^n 
\end{align*}
where for $x \in \mathbb{R}$, $(x)_k = x(x+1)\dots (x+k-1)$ is the Pochhammer symbol (\cite{Rainville}, p.45). Inverting the order of summation and identifying coefficients of $z^n$, it follows that 
\begin{equation*}
r_n = \frac{1}{2}\sum_{k=1}^n \frac{(-1/2)_k}{k!} \frac{(2k)_{n-k}}{(n-k)!} (-4\kappa^2)^k, \, n \geq 1.
\end{equation*}
Writing $(2k)_{n-k} = \Gamma(n+k)/\Gamma(2k), k \geq 1$, using Legendre duplication formula (\cite{Erd}) 
\begin{equation*}
\Gamma(2k) = 2^{2k-1}(k-1)!(1/2)_k, 
\end{equation*}
and since 
\begin{equation*}
 (-1/2)_k = -\frac{1}{2k-1}(1/2)_k, 
\end{equation*}
one gets 
\begin{align*}
r_n &= (n-1)!\sum_{k=1}^n \frac{(-1/2)_k}{(1/2)_k} \frac{(n)_k}{(n-k)!(k-1)!} \frac{(-\kappa^2)^k}{k!} 
\\& = -(n-1)!\sum_{k=1}^n \frac{1}{2k-1} \frac{(n)_k}{(n-k)!(k-1)!} \frac{(-\kappa^2)^k}{k!}
\\& =  -(n-1)!\sum_{k=0}^{n-1} \frac{1}{2k+1} \frac{(n)_{k+1}}{(n-1-k)!(k+1)!} \frac{(-\kappa^2)^{k+1}}{k!}
\\& = n \sum_{k=0}^{n-1} \frac{(1-n)_k}{2k+1} \frac{(n+1)_{k}}{(k+1)!} \frac{(\kappa^2)^{k+1}}{k!}
\\& = n\kappa \sum_{k=0}^{n-1} \frac{(1-n)_k}{2k+1} \frac{(n+1)_{k}}{(2)_k} \frac{(\kappa)^{2k+1}}{k!}.
\end{align*}
Finally
\begin{align*}
\frac{d}{d\kappa} \sum_{k=0}^{n-1} \frac{(1-n)_k}{2k+1} \frac{(n+1)_{k}}{(2)_k} \frac{(\kappa)^{2k+1}}{k!} &= \sum_{k=0}^{n-1} (1-n)_k\frac{(n+1)_{k}}{(2)_k} \frac{(\kappa)^{2k}}{k!}  
\\& = \frac{(n-1)!}{(2)_{n-1}} P_{n-1}^{1,0}(1-2\kappa^2) = \frac{1}{n}P_{n-1}^{1,0}(1-2\kappa^2)
\end{align*}
where the second equality follows from \cite{Rainville}, p.255.
\end{proof}

Now, we proceed to the study of $\mu_t$. 
\section{The time-dependent regime} 
   
\subsection{Time-dependent recursive relation}
This paragraph is devoted to the proof via free stochastic calculus of the following result:  
\begin{pro}\label{prop11}
Let 
\begin{equation*}
s_n(t) := e^{nt}\tau((SY_tSY_t^{\star})^n) = e^{nt}r_n(t), \, \, n \geq 1, 
\end{equation*}
then
\begin{eqnarray*}
s_1(t) &=& \kappa^2e^t + (1-\kappa^2). \\ 
 \partial_ts_n(t) &=& - n \sum_{j=1}^{n-1}s_j(t)s_{n-j}(t) + \kappa^2 n^2 e^{nt}, \,  n \geq 2. 
\end{eqnarray*}
\end{pro}
\begin{proof}: it goes along the same lines of that of Proposition 1 in  \cite{DHH}, with minor modifications due to the cancellations $S^2 = {\bf 1}$ rather than $P^2 = P$. For the reader's convenience, we write the whole proof and recall first Theorem 3.4 in \cite{BenL}:
\begin{teo}\label{Theo}
Let $n \geq 1$ and define
\begin{equation*}
f_{2n}(a_1, \dots, a_{2n}, t) := e^{nt} \tau(a_1Y_ta_2Y_t^{\star}\dots a_{2n-1}Y_ta_{2n}Y_t^{\star})
\end{equation*}
where $\{a_1, \dots, a_{2n}\} \in \mathscr{A}$ is $\star$-free with $Y$. Set $f_0(A,t) := \tau(A)$ for any $A \in \mathscr{A}$ then 
\begin{align*}
&\partial_t  f_{2n}(a_1, \dots, a_{2n}, t) = - \sum_{\substack{1 \leq k < l \leq 2n \\ l - k \equiv 0[2]}}f_{2n - (l-k)}(a_1, \dots, a_k, a_{l+1}, \dots, a_{2n}, t)f_{l-k}(a_{k+1}, \dots, a_l,t) +
\\&  e^t\sum_{\substack{1 \leq k < l \leq 2n \\ l-k -1\equiv 0[2]}}f_{2n - (l-k)-1}(a_1, \dots, a_{k-1}, a_k a_{l+1}, a_{l+2}, \dots, a_{2n}, t)f_{l-k-1}(a_la_{k+1}, a_{k+2}, \dots, a_{l-1},t)
\end{align*}
\end{teo}
Now, we specialize Theorem \ref{Theo} to $a_k = S$ for all $1 \leq k \leq n$ so that $f_{2n} = s_n$ and consider $n \geq 2$ (for $n=1$, the result is derived for instance from \cite{BenL}, p.923). Since both indices $k,l$ in the first (respectively second) sum in Theorem \ref{Theo} have the same (respectively different) parity, therefore $k$ and $l+1$ in the second sum have the same parity and so do $l$ and $k+1$. Accordingly, the first sum does not contain terms $f_0(\cdot, t)$ while the second does: they correspond to indices $l=2n, k=1$ and to $l=k+1, 1 \leq k \leq 2n-1$. Since $S^2 = {\bf 1}$ and since $\tau$ is a trace, then the contribution of indices $k=1, l=2n$ is 
\begin{equation*}
\kappa^2 e^{nt}.
\end{equation*} 
For $l=k+1, 1 \leq k \leq 2n-1$, we distinguish two cases: the contribution of $1 \leq k \leq 2n-2$ is
\begin{equation*}
(2n-2)\kappa^2e^{nt}
\end{equation*}
while that of $k=2n-1, l=2n$ is $\kappa^2e^{(n-1)t}$. Thus, the whole contribution of the indices $k=1,l=2n$ and of $1 \leq k \leq 2n-1, l=k+1$ is 
 \begin{equation}\label{C0}
2n\kappa^2e^{nt}. 
 \end{equation}
Next we write $l = k+2s+1$ for integer positive values of $s$ and distinguish $n = 2$ and $n \geq 3$. If $n = 2$ then there is no additional term in the second sum, while if $n \geq 3$ we separate $k=1$ and $2 \leq  k  \leq 2n-3$. By the same properties of $a, \tau$ mentioned above, the contribution of indices $k=1, l=2s+2, 1 \leq s \leq n-2$ is 
 \begin{equation}\label{C1}
(n-2)\kappa^2e^{nt}. 
 \end{equation}
 For the remaining values of $2 \leq k \leq 2n-3$, we distinguish even and odd ones: the contribution of indices $k=2j, 1 \leq j \leq n-2, l = 2j+2s+1$ is 
 \begin{equation}\label{C2}
 \sum_{j=1}^{n-2}\sum_{s=1}^{n-j-1}\kappa^2e^{nt} = \kappa^2e^{nt} \frac{(n-1)(n-2)}{2},
 \end{equation}
while for $k=2j+1, 1 \leq j \leq n-2, l=2s+2j+2$ we distinguish $1 \leq s \leq n-j-2$ and $s = n-j-1$. When $1 \leq s \leq n-j-2$ we get  
 \begin{equation}\label{C3}
 \sum_{j=1}^{n-2}\sum_{s=1}^{n-j-2}\kappa^2e^{nt} = \kappa^2e^{nt} \frac{(n-2)(n-3)}{2} 
 \end{equation}
while for $s =  n-j-1$ we get 
 \begin{equation}\label{C4}
 \sum_{j=1}^{n-2}\kappa^2 e^{nt}  = (n-2)\kappa^2e^{nt}. 
 \end{equation}
Coming to the first sum, its contribution is the same as in \cite{DHH}, Lemma 1:
\begin{equation}\label{C5}
-n\sum_{k=1}^{n-1}s_{n-k}(t) s_k(t).
\end{equation}
The proposition is proved after summing \eqref{C0}, \eqref{C1}, \eqref{C2}, \eqref{C3} and \eqref{C4}. 
\end{proof}

\subsection{Dynamics of the Herglotz transform}
Here, we transform the time-dependent recursive relation into a pde governing Herglotz transform $z\mapsto H(t,z)$ of $\mu_t$. Recall that
$$
H(t,z)=\int_{\mathbb{T}}\frac{\xi+z}{\xi-z}d\mu_t(\xi)
$$
and that the moments
\begin{equation*}
r_n(t) = \tau[(SY_tSY_t^{\star})^n], \, n \geq 1
\end{equation*}
are the coefficients of the expansion of $H(t, \cdot)$ as an analytic function: 
\begin{equation*}
H(t,z) = 1+2\sum_{n \geq 1}r_n(t)z^n, \quad |z| < 1. 
\end{equation*} 
Using Proposition \ref{prop11}, we readily get:
\begin{pro}
The Herglotz transform $H$ satisfies the equation
\begin{equation}\label{transport}
\partial_t H+\frac{z}{2}\partial_z H^2=2\kappa^2\frac{z(1+z)}{(1-z)^3},\quad H(0,z)=\frac{1+z}{1-z},\quad\,|z|<1.
\end{equation}

\end{pro}
\begin{proof}
Elementary computations show the sequence  $(r_n(t))_{n\geq1}$ satisfies
\begin{eqnarray*}
\partial_tr_1(t) &=& {-r_1(t)} + \kappa^2,\\ 
 \partial_tr_n(t) &=& - nr_n(t) - n \sum_{j=1}^{n-1}r_j(t)r_{n-j}(t) + \kappa^2 n^2, \,  n \geq 2. 
\end{eqnarray*}
Consequently, 
\begin{eqnarray*}
\partial_t H&=&2\sum_{n\geq1}\partial_t r_n(t) z^n\\
&=&2\kappa^2\sum_{n\geq1}n^2 z^n -2 \sum_{n\geq1} nr_n(t) z^n-2\sum_{n\geq2}n \sum_{j=1}^{n-1}r_j(t)r_{n-j}(t) z^n\\
&=&2\kappa^2\frac{z(1+z)}{(1-z)^3} -z \partial_z H-2\sum_{j\geq1}r_j z^j \sum_{n\geq j+1}nr_{n-j}(t) z^{n-j}\\
&=&2\kappa^2\frac{z(1+z)}{(1-z)^3} -z\partial_z H-4\frac{H-1}{2}\sum_{j\geq1}jr_{j}(t) z^{j}\\
&=&2\kappa^2\frac{z(1+z)}{(1-z)^3} -z H\partial_z H.
\end{eqnarray*}
\end{proof}

\begin{rem}
The equation \eqref{transport} is a non homogeneous Burgers equation. It allows to retrieve the expression of $H_{\infty}$ already obtained in the previous section. Indeed, any stationary solution of 
\eqref{transport} is  a solution of $\partial_tH = 0,$ that is $H(t,z) = H(z)$ solves the first-order ordinary differential equation
\begin{equation*}\label{station1}
\partial_z(H^2) = 4\kappa^2 \frac{1+z}{(1-z)^3}. 
\end{equation*}
After integrating and taking into account $H(0) = 1$, we get
\begin{equation*}\label{station2}
H^2(z) = 4\kappa^2 \frac{z}{(1-z)^2} + H^2(0) = H_{\infty}^2(z).
\end{equation*}
\end{rem}

\section{Resolution of the Burgers equation}
In the sequel, we prove that the dynamics of the Herglotz transform is completely determined by the initial condition $H(0,\cdot)$, the long-time behavior $H_\infty$ and some characteristics curves. To make easier the computations, we first use the M\"obius transform 
 \begin{equation*}
 z\mapsto y=\frac{1+z}{1-z}
 \end{equation*}
 which realizes a one-to-one map between the open unit disc and the right half-plane $\{\Re z>0\}$.  Indeed, this transform replaces the fraction in the RHS of \eqref{transport} by a cubic polynomial. To see this, set
 $$
 F(t,y) \triangleq H({t}, z),\quad y=\frac{1+z}{1-z}.
 $$
 Then $F$ satisfies the equation
 \begin{equation}\label{transp23}
 \partial_t F+\frac14(y^2-1)\partial_y F^2=\frac{\kappa^2}{2} y(y^2-1), \quad F(0,y)=y.
 \end{equation}
 Observe that the stationary solution $H_\infty$ reads after this variable change
 $$
 H_\infty(z) \triangleq F_\infty(y) = \sqrt{(1-\kappa^2)+\kappa^2 y^2}.
 $$
 The major step toward the proof of our main result is the following theorem:

\begin{teo}\label{teo65}
Let $F$ be a solution of  the nonlinear equation \eqref{transp23}. Then 
\begin{equation}\label{trans67}
F^2(t,\phi(t,y))=\kappa^2\phi^2(t,y)+(1-\kappa^2)y^2, \quad \Re y>0,
\end{equation}
 with $$ \phi^2(t,y)= 1-\frac{4a\lambda e^{\sqrt{a}t}}{(b+  \lambda e^{\sqrt{a} t})^2-4\kappa^2 a}
 $$
where 
 $$
 a\triangleq \kappa^2+(1-\kappa^2)y^2,\, b = a+\kappa^2, \quad  \lambda \triangleq \frac{(1-\kappa^2)^2y^2}{(1+\sqrt{a})^2}(1-y^2).
 $$
 In the z-configuration this reads
 \begin{equation}\label{DynBis}
 H^2(t,\psi(t,z))=\kappa^2\left(\frac{1+\psi(t,z)}{1-\psi(t,z)}\right)^2+(1-\kappa^2)\left(\frac{1+z}{1-z}\right)^2, 
 \end{equation}
 with
 $$
\psi(t,z)=\frac{\phi(t,y)-1}{\phi(t,y)+1}\quad\hbox{and}\quad z=\frac{y-1}{y+1}\cdot
 $$
 \end{teo}

Before going into the details of the proof, some remarks are in order.
\begin{rems} \noindent
\begin{itemize}
\item For $\Re y>0,$ the trajectory $t\mapsto \phi(t,y)$ may cease to exist at some blow-up time, say $T_\star(y).$ This corresponds in the $z$ configuration to the time when the curve crosses  the unit circle at  the singularity $z=1$.  
\item Since the Herglotz transform $H(t,\cdot)$ is analytic inside the unit circle, the function $F(t, \cdot)$ is analytic in $\{\Re y>0\}$ and   the identity \eqref{trans67}  is meaningful  wherever the trajectory $t\mapsto \phi(t,y)$ does not leave this region. As the square root is well-defined there then knowing $\phi^2$ is sufficient to know $\phi.$  We point out that in general the curves $\phi$ cross the imaginary axis but there are some cases where the curves remain  confined therein as for instance when $y\in]0,1]$.
\item Observe that $\phi(0,y)=y$ so that  \eqref{trans67} can be rewritten under the form
$$
\widetilde{F}(t,\phi(t,y))=\widetilde{F}(0,y), \quad \widetilde{F}(t,y) \triangleq F(t,y)-\kappa^2 y^2.
$$
This means that the function $\widetilde{F}$ is constant along the trajectories. 
\item Keeping in mind the expressions of $H_{\infty}$ and $H(0,\cdot)$, \eqref{DynBis} is easily seen to be equivalent to \eqref{Dyn}. 
\item The square root of $a$ is well defined on $\{\Re y>0\}$.
\end{itemize}
\end{rems}
\begin{proof}
Let $\phi$ be the solution of the ordinary differential equation (hereafter ODE)
\begin{equation}\label{cond1}
\partial_t\phi=\frac12(\phi^2-1)F(t,\phi),\quad \phi(0,y)=y.
\end{equation}
Then differentiating the function $F_1: (t,y)\mapsto F(t,\phi(t,y))$ with respect to $t$ yields
$$
\partial_t F_1(t,y)=\frac12\kappa^2\phi(t,y)\left(\phi^2(t,y)-1\right).
$$
Therefore solving the pde \eqref{transp23} reduces to the study of the two coupled ODEs:
\begin{equation}
    \label{coupled}
\left\{ \begin{array}{ll}
\partial_t\phi=\frac12(\phi^2-1)F_1,\\
\partial_t F_1=\frac12\kappa^2\phi\left(\phi^2-1\right),\\
\phi(0,y)=y, F_1(0,y)=y.
\end{array} \right.
\end{equation}

It is clear that \eqref{coupled} entails
$$
F_1\partial_t F_1-\kappa^2\phi\partial_t\phi=0.
$$
Hence, integrating with respect to $t$ yields \eqref{trans67}
\begin{equation*}
F_1^2(t,y)-\kappa^2\phi^2(t,y)=\left(1-\kappa^2\right)y^2.
\end{equation*}
which in turn leads to 
\begin{equation}
    \label{first1}
\left\{ \begin{array}{ll}
\partial_t\phi =\frac12(\phi^2-1)\sqrt{(1-\kappa^2)y^2+\kappa^2\phi^2},\\
\phi(0,y)=y.
\end{array} \right.
\end{equation}
Now, we shall solve \eqref{first1} for fixed $y > 0$ which is equivalent to $z \in (-1,1)$. First, observe that $\phi(t,y) = \pm1$ are stationary solutions of \eqref{first1} for any $\kappa$ and by uniqueness of the solution, we deduce that if $0 < y <1$ then $\phi$ is global in time and\footnote{If $y=1$ then $\phi \equiv 1$.} 
$$
|\phi(t,y)| <  1, \forall t\in \mathbb{R}_+.
$$
More precisely, assume for instance that there exists $T> 0$ and $0 < y < 1$ such that $\phi(T,y) = 1$. Then $\phi$ and $t \mapsto 1$ solve the Cauchy problem corresponding to the data $\phi(T,y) =1$. Necessarily, $\phi=1$ which is in contrast with $0 < y < 1$. As a matter of facy, $t \mapsto \phi(t,y)$ is non increasing and crosses the right half-plane with limit $-1$ as $t \rightarrow \infty$. Similar arguments show that 
$$
y > 1\Rightarrow \phi(t,y) >  1,\forall t\in]0, T^\star[.
$$
Indeed, the lifespan $T^{\star}$ is finite due to the cubic power of the non linearity in \eqref{first1}. Next, we need to compute the indefinite integral
\begin{equation*}
2\int \frac{dx}{(1-x^2)\sqrt{(1-\kappa^2)y^2 + \kappa^2x^2}} 
\end{equation*}
for real positive $x$ (which is equivalent to $\phi \in (-1,1)$). First, we perform the variable change $u = 1-x^2 \in (-\infty, 1)$ to transform the integral to 
\begin{equation*}
-\int \frac{du}{u\sqrt{a-bu + cu^2}}
\end{equation*}
where we set 
\begin{eqnarray*} 
a &=& (1-\kappa^2)y^2+\kappa^2 \\
b &=& (2\kappa^2 + (1-\kappa^2)y^2) \\ 
c &=&  \kappa^2.
\end{eqnarray*}
Note that $b^2 - 4ac = (1-\kappa^2)^2y^4$ and that the roots of $a-bu+cu^2 = 0$ lie in $[1,\infty]$. Next, we perform the variable change 
\begin{equation*}
\sqrt{a}(1-vu) = \sqrt{a-bu + cu^2}
\end{equation*}
and we easily get 
\begin{equation*}
u = \frac{2av-b}{av^2-c}, \quad du = -2a \frac{av^2 - bv+ c}{(av^2-c)^2} dv. 
\end{equation*}
As a result 
\begin{equation*}
\int \frac{du}{u\sqrt{a-bu + cu^2}} = 2\sqrt{a}\int \frac{dv}{2av-b} = \frac{1}{\sqrt{a}}\ln\left|\frac{2a - bu - 2\sqrt{a}\sqrt{a-bu + cu^2}}{u}\right|.
\end{equation*}
But $u < 1$ so that $2a-bu > 2a-b = (1-\kappa^2)y^2 > 0$ and 
\begin{equation*}
(2a-bu)^2 - 4a(a-bu+cu^2) = (b^2-4ac)u^2 > 0. 
\end{equation*}
Consequently 
\begin{equation*}
\int \frac{du}{u\sqrt{a-bu + cu^2}}  = \frac{1}{\sqrt{a}}\ln\frac{2a - bu - 2\sqrt{a}\sqrt{a-bu + cu^2}}{|u|}.
\end{equation*}
and if $U(t,y) \triangleq 1-\phi^2(t,y),$ then 
\begin{equation*}
\frac{1}{\sqrt{a}}\ln\frac{2a - bU(t,y) - 2\sqrt{a}\sqrt{a-bU(t,y) + cU^2(t,y)}}{|U(t,y)|} =  t + A
\end{equation*}
for some $A = A(y,\kappa)$. Equivalently, 
\begin{equation*}
\frac{2a - bU(t,y) - 2\sqrt{a}\sqrt{a-bU(t,y) + cU^2(t,y)}}{|U(t,y)|} =  \lambda e^{\sqrt{a}t}
\end{equation*} 
where $\lambda = e^{\sqrt{a}A}$. Writing this equality as 
\begin{equation*}
2a - [b +\epsilon \lambda e^{\sqrt{a}t}] U(t,y) = 2\sqrt{a}\sqrt{a-bU(t,y) + cU^2(t,y)}
\end{equation*} 
and raising it to the square we get 
\begin{equation}\label{ident23}
\Big(\big(b +\epsilon  \lambda e^{\sqrt{a}t}\big)^2 -4ac\Big)U(t,y) = 4a\big(b +\epsilon  \lambda e^{\sqrt{a}t}\big)- 4ab = 4\epsilon  a\lambda e^{\sqrt{a}t},
 \end{equation}
 $\epsilon \in\{-1,1\}$ being the sign of $U.$ From the observation made before the sign of $U$ does not change through the time. To find the value of $\lambda$ we check the preceding equation for $t=0$
 $$
  \lambda^2+2\epsilon(b-\frac{2a}{1-y^2})\lambda+b^2-4ac=0.
 $$
 
 {$\bullet$ \it Case $y^2\le1$:} This corresponds to $\epsilon=1$ and the above equation becomes  
 \begin{equation}\label{eqs212}
 \lambda^2+2(b-\frac{2a}{1-y^2})\lambda+b^2-4ac=0.
\end{equation}
 The discriminant of this polynomial 
 \begin{eqnarray*}
 \Delta&=&\frac{16a}{(1-y^2)^2}\left(a-b(1-y^2)+c(1-y^2)^2\right)\\
 &=&\frac{16a y^2}{(1-y^2)^2}\left(b-2c+cy^2\right)\\
 &=&\frac{16a y^4}{(1-y^2)^2}.
 \end{eqnarray*}
 
 Therefore the only solution of \eqref{eqs212} which is not singular at $y=1$ is
 \begin{eqnarray*}
 \lambda&=&-b+\frac{2a}{1-y^2}-\frac{2 y^2\sqrt{a}}{1-y^2}\\
 &=&-b+2+\frac{2(a-1)}{1-y^2}-\frac{2(\sqrt{a}-1)}{1-y^2}\\
 &=& -b+2\kappa^2+2(1-\kappa^2)\frac{y^2}{1+\sqrt{a}}\\
 &=&(1-\kappa^2)y^2\left(-1+\frac{2}{1+\sqrt{a}}\right)\\
 &=&\frac{(1-\kappa^2)^2y^2}{(1+\sqrt{\kappa^2+(1-\kappa^2)y^2})^2}(1-y^2).
  \end{eqnarray*}
 
  {$\bullet$ \it Case $y^2\geq1 $:}  reproducing the same computation yields to
  $$
  \lambda=\frac{(1-\kappa^2)^2y^2}{(1+\sqrt{\kappa^2+(1-\kappa^2)y^2})^2}(y^2-1).
  $$
  
  Hence in both cases we get
  $$
  \epsilon\lambda=\frac{(1-\kappa^2)^2y^2}{(1+\sqrt{\kappa^2+(1-\kappa^2)y^2})^2}(1-y^2).
  $$
  
  Finally, \eqref{ident23} yields
  \begin{eqnarray*}
  U(t,y)&=&\frac{4a\lambda e^{\sqrt{a}t}}{\big(b+ \lambda e^{\sqrt{a}t}\big)^2-4ac} = 1 - \phi^2(t,y)
  \end{eqnarray*}
 where we simply write $\lambda$ instead of  $\epsilon \lambda$ and hope there is no ambiguity. 
  \end{proof}
  
Now we discuss some  properties related to the monotonicity of the flow map.
\begin{pro}\label{pro987}
Let $1<y_1<y_2$ and denote by $T^\star_i$ the lifespan of the trajectory $t\mapsto \phi(t,y_i)$. Then $T^\star_2< T^\star_1$ and
$$
1<\phi(t,y_1)< \phi(t,y_2),\quad \forall t\in[0,T^\star_2[.
$$
\end{pro}
\begin{proof}
The monotonicity of the flow is a consequence of the comparison principle. Indeed, denote by $t\mapsto \phi_i(t)$ the trajectory associated to $y_i$.  Since $\phi_i>1$, then 
\begin{equation}
    \label{first123}
\left\{ \begin{array}{ll}
\partial_t\phi_2 = \displaystyle \frac{1}{2}(\phi_2^2-1)\sqrt{(1-\kappa^2)y_1^2+\kappa^2\phi_2^2},\\
\phi_2(0)>\phi_1(0).
\end{array} \right.
\end{equation}
Assume that the curves of $\phi_1$ and $\phi_2$ can intersect and let $T$ be the first time of intersection. This means that
$$
\forall t\in [0,T[,\,\phi_1(t)<\phi_2(t), \quad \hbox{and}\quad \phi_1(T)=\phi_2(T).
$$
Thus necessary $\phi^{\prime}_2(T)\le\phi^\prime_1(T),$ but from the differential inequality \eqref{first123} we observe that $\phi^\prime_2(T)>\phi^\prime_1(T)$, which is impossible. The inequality between the lifespans follows easily from the blow up criterion.
\end{proof}

\section{Proof of Theorem \ref{THM} and Corollary \ref{COR}} 
After this long wave of computations, we proceed to the proof of theorem \ref{THM} and of its corollary. It consists of two lemmas: the first one gives the limit of the RHS in \eqref{Limit} as $n \rightarrow \infty$. The second one shows that for any fixed time $t$, the flow $y \mapsto \phi(t,y)$ blows up at some real $y_t$. Equivalently, there exists a real $z_t$ such that $\psi(t,z_t) = 1$.  
\begin{lem}
The following assertions hold true:
\begin{itemize}
\item 
Let $\phi \in [0,2\pi]$, then 
\begin{equation*}
\frac{1}{2^{2n}}\sum_{k=1}^n \binom{2n}{n-k}\left(e^{ik\phi} + e^{-ik\phi}\right) =\cos^{2n}(\phi/2) -  \frac{1}{2^{2n}} \binom{2n}{n}.
\end{equation*}
\item Let $\mu$ be a probability distribution on the unit circle $\mathbb{T}$, then
$$
\lim_{n \rightarrow \infty} \frac{1}{2^{2n}}\sum_{k=1}^n \binom{2n}{n-k}\int_{\mathbb{T}}(z^k+\overline{z}^k)\mu(dz)=\mu(\{1\}).
$$
\item Recall the spectral distribution $\mu_t$ of the unitary operator $SY_t SY_t^\star$. Then
$$
\lim_{n\to+\infty}\tau[(PY_tPY_t^{\star})^n]=\frac{1}{2}[2\theta-1+\mu_t(\{1\})]. 
$$
\end{itemize}
\end{lem} 

\begin{proof} \noindent
\begin{itemize}
\item Using the fact that 
\begin{equation*}
\binom{2n}{n-k} = \binom{2n}{n+k}
\end{equation*}
we write
\begin{align*}
\sum_{k=1}^n \binom{2n}{n-k}\left(e^{ik\phi} + e^{-ik\phi}\right) &= \sum_{k=1}^n \binom{2n}{n-k}e^{ik\phi} + \sum_{k=-n}^{-1}\binom{2n}{n+k}e^{ik\phi}
\\& = \sum_{k=-n}^n \binom{2n}{n+k}e^{ik\phi} - \binom{2n}{n}
\\& = e^{-in\phi}\sum_{k=0}^{2n} \binom{2n}{k}e^{ik\phi} - \binom{2n}{n}\\
&= 2^{2n}\cos^{2n}(\phi/2) - \binom{2n}{n}.
\end{align*}
\item Identifying $\mu$ with its image under the map $z \mapsto \arg(z) \in (-\pi, \pi)$ we readily get
\begin{eqnarray*}
 \frac{1}{2^{2n}}\sum_{k=1}^n \binom{2n}{n-k}\int_{\mathbb{T}}(z^k+\overline{z}^k)\mu(dz)&=& \frac{1}{2^{2n}}\sum_{k=1}^n \binom{2n}{n-k}\int_{-\pi}^{\pi}(e^{ik\phi}+e^{-ik\phi})\mu(d\phi)\\
 &=&\int_{\mathbb{T}}\cos^{2n}(\phi/2)\mu(d\phi) - \frac{1}{2^{2n}}\binom{2n}{n}.
\end{eqnarray*}
The result follows from Stirling formula
\begin{equation*}
\lim_{n \rightarrow \infty}\frac{1}{2^{2n+1}}\binom{2n}{n} = \lim_{n \rightarrow \infty}\frac{1}{\sqrt{\pi n}} = 0.
\end{equation*}
\item Due to the trace property of $\tau$, the spectral distributions of $SY_tSY_t^{\star}$ and of $Y_t^{\star}SY_tS$ coincide so that $\mu_t$ is invariant under $z \mapsto \overline{z}$. Hence 
\begin{equation*}
 2\int_{\mathbb{T}}z^k\mu_t(dz) = \int_{\mathbb{T}}(z^k+\overline{z}^k)\mu_t(dz)
 \end{equation*}
  and the desired limit follows from \eqref{Limit}.
\end{itemize}
\end{proof}

Recall from the previous section that $\phi$ is global in time if and only if $ 0 < y \le 1$. However, these values of $y$ correspond to $-1 < z \leq 0$ while we need to reach $z=1$ along the real interval $(0,1]$.  When $y \geq 1$, we shall prove the following \begin{lem}

\begin{enumerate} \noindent
\item For any $t>0$, there exists $z_t\in ]0,1[$ such that $\psi(t, z_t)=1$, where $\psi$ is the flow defined in Theorem $\ref{teo65}.$
\item For any $t > 0$,
$$
\lim_{z\to 1, z < 1}(1-z)H(t,z)=2|\kappa|.
$$
\end{enumerate}
\end{lem}
\begin{proof}
$(1)$ Let $t>0$, then $\psi(t,z_t) = 1$ is equivalent to 
$$
|\phi(t,y_t)|=+\infty, \quad y_t = \frac{1+z_t}{1-z_t}.
$$
With regard to the expression of $\phi$ in Theorem \ref{teo65}, we seek $y=y_t$ such that 
\begin{equation*}
(b-\lambda e^{\sqrt{a}t}\big)^2-4ac = 0
\end{equation*}
where we recall that $a = (1-\kappa^2)y^2+\kappa^2, b = \kappa^2 + a, b^2-4ac = (1-\kappa^2)^2y^4$ and
\begin{equation*}
\lambda = \frac{(1-\kappa^2)^2y^2}{(1+\sqrt{a})^2}(y^2-1).
\end{equation*}
Note by passing that when $\kappa = 0$ one has $c= 0, b = a = y^2, \lambda = (y-1)/(y+1)$ so that 
\begin{equation*}
(b-\lambda e^{\sqrt{a}t}\big)^2-4ac = 0 \Leftrightarrow \frac{y-1}{y+1}e^{yt} = 1\Leftrightarrow ze^{t(1+z)/(1-z)} = 1.
\end{equation*}
The last equality states that the $\Sigma$-transform of $Y_{2t}$ attains the value $1$ (see \cite{Biane}). For general values of $\kappa \in (-1,1)$, we are led to
\begin{equation*}
(1-\kappa^2)^2y^4 + \frac{(1-\kappa^2)^4y^4}{(1+\sqrt{a})^4}(y^2-1)^2e^{2\sqrt{a}t} - 2\frac{(1-\kappa^2)^2y^2}{(1+\sqrt{a})^2}(y^2-1)(\kappa^2+a)e^{\sqrt{a}t} = 0.
\end{equation*}
Equivalently 
\begin{equation*}
(1+\sqrt{a})^4 y^2 + (1-\kappa^2)^2y^2(y^2-1)^2e^{2\sqrt{a}t} - 2(1+\sqrt{a})^2(y^2-1)(\kappa^2+a)e^{\sqrt{a}t} = 0.
  \end{equation*}  
But since 
\begin{equation*}
y^2 = \frac{a-\kappa^2}{1-\kappa^2}, 
\end{equation*}
  then 
  \begin{equation*}
(1+\sqrt{a})^4\frac{a-\kappa^2}{1-\kappa^2} + (1-\kappa^2)^2\frac{a-\kappa^2}{1-\kappa^2} \left(\frac{a-1}{1-\kappa^2}\right)^2e^{2\sqrt{a}t} - 2(a+\kappa^2)(1+\sqrt{a})^2 \frac{a-1}{1-\kappa^2}e^{\sqrt{a}t} = 0
\end{equation*}
which simplifies to   
   \begin{equation*}
(1+\sqrt{a})^4(a-\kappa^2) + (a-\kappa^2)(a-1)^2e^{2\sqrt{a}t} - 2(a+\kappa^2)(1+\sqrt{a})^2 (a-1)e^{\sqrt{a}t} = 0.
\end{equation*}
Writing $(a-1) = (\sqrt{a}-1)(\sqrt{a}+1)$ we simplify further to get 
\begin{equation*}
(1+\sqrt{a})^2(a-\kappa^2) + (a-\kappa^2)(\sqrt{a}-1)^2e^{2\sqrt{a}t} - 2(a+\kappa^2) (a-1)e^{\sqrt{a}t} = 0.
\end{equation*}
Gathering terms proportional to $\kappa^2$, we are led to 
\begin{equation*}
\kappa^2[(1+\sqrt{a})^2 + (\sqrt{a}-1)^2e^{2\sqrt{a}t} + 2(a-1)e^{\sqrt{a}t}] = a[(1+\sqrt{a})^2 + (\sqrt{a}-1)^2e^{2\sqrt{a}t} - 2(a-1)e^{\sqrt{a}t}]
 \end{equation*}
and finally to 
\begin{equation*}
\kappa^2\left[\frac{\sqrt{a}+1}{\sqrt{a}-1} + e^{\sqrt{a}t}\right]^2 = a\left[\frac{\sqrt{a}+1}{\sqrt{a}-1} - e^{\sqrt{a}t}\right]^2.
 \end{equation*}
Hence either 
\begin{equation}\label{poss1}
e^{\sqrt{a}t}[|\kappa| +\sqrt{a}] =  \frac{\sqrt{a}+1}{\sqrt{a}-1}[\sqrt{a} - |\kappa|]
\end{equation}
or 
\begin{equation}\label{poss2}
e^{\sqrt{a}t}[\sqrt{a} - |\kappa|] =  \frac{\sqrt{a}+1}{\sqrt{a}-1}[\sqrt{a} + |\kappa|].
\end{equation}
For fixed $t>0$ each equation admits a unique solution  $a>1$ and to select the suitable solution we refer to the special case $\kappa=1$. Naturally, the dependence of the solution with respect to the parameter $\kappa$ should be continuous but as we will see  this property is violated by the second formula at the endpoint $\kappa=1$. Indeed, coming back to \eqref{first1} we see that the solution is given by 
$$
\frac{\phi^2(t,y)-1}{\phi^2(t,y)}=\frac{y^2-1}{y^2}e^t.
$$
Consequently,
$$
U(t,y)= 1- \phi^2(t,y) = \frac{(y^2-1)e^t}{(y^2-1) e^t-y^2},\quad y\geq1
$$
which becomes singular when
\begin{equation}\label{poss3}
e^t=\frac{y^2}{y^2-1}\cdot
\end{equation}
But the left-hand side of \eqref{poss2} converges to zero when $\kappa$ goes to $1$, whereas  the right-hand side diverges at this point and thus this formula is not valid for this limiting case. Coming to the formula \eqref{poss1}, l'Hopital rule leads to
\begin{eqnarray*}
\lim_{\kappa\to 1}\frac{\sqrt{a}+1}{\sqrt{a}+|\kappa|}\frac{\sqrt{a} - |\kappa|}{{\sqrt{a}-1 }}&=&\lim_{\kappa\to 1}\frac{\sqrt{a} - \kappa}{{\sqrt{a}-1 }} = \frac{y^2}{y^2-1}\cdot
\end{eqnarray*}
This gives exactly the equation \eqref{poss3} and thus we retain \eqref{poss1}. 

${\bf{2)}}$ Fix $t>0$, there exists $y_t>1$ such that 
$$\lim_{y\to y_t^-}\phi(t,y)=+\infty.
$$
We point out from Proposition \ref{pro987} that the lifespan of the trajectory starting at $y\in]1,y_t[$ is larger than the lifespan $t$ of  the curve $\tau\mapsto \phi(\tau,y_t)$. Keeping in mind $F(t,y) = H(t,z)$ where $y = (1+z)/(1-z)$, we get
\begin{eqnarray*}
\lim_{z\to 1^-}(1-z) H(t,z)&=&2\lim_{y\to +\infty}\frac{F(t,y)}{y} = 2\lim_{y\to y_t^-}\frac{F(t,\phi(t,y))}{\phi(t,y)}.
\end{eqnarray*}
But formula \eqref{trans67} entails
\begin{eqnarray*}
\lim_{y\to y_t^-}\frac{F^2(t,\phi(t,y))}{\phi^2(t,y)}=\kappa^2
\end{eqnarray*}
whence we deduce
$$
\lim_{z\to 1^-}(1-z) H(t,z)=2|\kappa|,
$$
the lemma is proved.
\end{proof}

\begin{proof}[Proof of theorem \ref{THM} and corollary \ref{COR}]
It is a general fact that the singular discrete part of $\mu_t$ corresponds exactly to the poles of $H(t,\cdot)$. Moreover, the weight that $\mu_t$ assigns to a given pole can be recovered using radial limits. in particular, $\mu_t$ assigns the weight $|\tau|$ to $z=1$ (\cite{CMR}). Theorem \ref{THM} then follows after taking the limit as $n \rightarrow \infty$ in \eqref{Limit} and using the last assertion in the first Lemma of this section. Corollary \ref{COR} follows readily after normalizing \eqref{Limit} by $1/\tau(P) = 1/\theta$ and taking the limit as $n \rightarrow \infty$: the LHS of \eqref{Limit} now tends to the weight that the spectral distribution of $J_t$ in the compressed space, say $\nu_t$, assigns to $x=1$ since 
\begin{equation*}
\frac{1}{\tau(P)}\tau[(PY_tPY_t^{\star})^n] = \int_0^1 x^n\nu_t(dx). 
\end{equation*}
 \end{proof}

\section{Analysis of the moments}
Set $\epsilon \triangleq \kappa^2 = [\tau(S)]^2$ and recall from Proposition \ref{prop11} that $s_n(t) = e^{nt}r_n(t), n \geq 1$ satisfy 
\begin{eqnarray*}
 \partial_ts_n(t) &=& - n \sum_{j=1}^{n-1}s_j(t)s_{n-j}(t) + \epsilon n^2 e^{nt} \\ 
 s_1(t) &=& \kappa^2e^t + (1-\kappa^2). 
\end{eqnarray*}
Recall also from \cite{DHH} that when $\epsilon = 0$ then 
\begin{equation*}
s_n(t) = \frac{1}{n}L_{n-1}^{(1)}(2nt) 
\end{equation*}
where $L_n^{(1)}$ is the $n$-th Laguerre polynomial (\cite{Rainville}). In this section, we perform an analysis of the `leading term' in $s_n(t)$ when $\epsilon \in [0,1)$, that is the term corresponding to the fastest decay $e^{-nt}$ of $r_n(t)$. Indeed, it is easy to see by induction that 
\begin{equation*}
s_n(t) = P_n(\epsilon, t) + \sum_{k=1}^{n}e^{kt}\dots,
\end{equation*}
where $P_n(\epsilon, t)$ is a polynomial in the variable $t$ whose degree is $n-1$. To proceed, we first observe that compared to the equation satisfied by $s_n(t)$ when $\epsilon = 0$, the deformation comes with the factor $e^{nt}$. Hence, the polynomials $P_n(\epsilon, t), n \geq 1$ still satisfy   
\begin{eqnarray*}
\partial_t P_n(\epsilon, t) &=& -n\sum_{k=1}^{n-1}P_j(\epsilon, t)P_{n-j}(\epsilon, t), \quad n \geq 2 \\
P_1(\epsilon, t) &=& (1-\epsilon).
\end{eqnarray*}
However, one needs to compute $P_n(\epsilon, 0)$ in order to determine the polynomials $P_n(\epsilon, t), n \geq 2$. We shall see in the sequel that while $P_n(0,0) = s_n(0) = 1$, $P_n(\epsilon, 0)$ changes dramatically when $\epsilon \in (0,1)$. To this end, we make use of the formula (\cite{NS}, Theorem 14.4)
\begin{equation*}
s_n(t) = \tau((SY_tSY_t^{\star})^n) = \sum_{\pi \in \textrm{NC(2n)}}c_{\pi}(S,S\dots,S)m_{K(\pi)}(Y_t,Y_t^{\star}, \dots, Y_t, Y_t^{\star}).
\end{equation*}
Here $\textrm{NC(2n)}$ is the lattice of non crossing partitions of size $2n$, $K(\pi) \in NC(2n)$ denotes the Kreweras complement of $\pi$, $c_{\pi}$ is the free cumulant of the $2n$-tuple $(a,a,\dots, a)$ associated with $\pi$ and $m_{K(\pi)}$ is the mixed moment of the $2n$-tuple $(Y_t,Y_t^{\star}, \dots, Y_t, Y_t^{\star})$ (\cite{NS}, Chapter IX). Now, the polynomial $P_n$ comes without any exponential factor, thus we only need to focus exactly on partitions $\pi \in \textrm{NC}(2n)$ such that their Kreweras complements $K(\pi)$ are non-parity alternating, that is each block of $K(\pi)$ lies either in $\{1, 3, \dots, 2n-1\}$ or in $\{2, 4, \dots, 2n\}$ (we identify $K(\pi) \approx \{1,2,3, \dots, 2n\}$). 
Indeed, the $k$-th moment of $Y_t$ is given by 
\begin{equation*}
e^{-kt/2}\frac{1}{k}L_{k-1}^{(1)}(kt), \,\, k \geq 1, 
\end{equation*}
so that the polynomial $P_n$ corresponds to $m_{K(\pi)}$ for which there is no cancellation between $Y$ and $Y^{\star}$.
According to \cite{NS} Exercise 9.42 p.153-154, the partition $\pi$ runs over the set $\textrm{NCE}(2n)$ of non crossing even partitions (each block of $\pi$ has even number of elements). Besides, since the constant term of 
\begin{equation*}
\frac{1}{k}L_{k-1}^{(1)}(kt)
\end{equation*}
equals $1$ for any $k \geq 1$ then we end up with
\begin{equation*}
P_n(\epsilon, 0) = \sum_{\pi \in \textrm{NCE}(2n)} c_{\pi}(S,S,\dots,S). 
\end{equation*}   
We can write this sum as 
\begin{equation*}
P_n(\epsilon, 0) = \sum_{\pi \in \textrm{NC}(2n)} \frac{1}{2^{|\pi|}}c_{\pi}(a_1-a_2,a_1-a_2,\dots,a_1-a_2)
\end{equation*}
where $a_1, a_2 \in \mathscr{A}$ are two free copies of $S$ and $|\pi|$ is the number of blocks of $\pi$. Indeed, by freeness of $a_1$ and $a_2$ and multi-linearity of free cumulants, one has 
\begin{equation*}
c_{V}(a_1-a_2, a_1-a_2, \dots, a_1-a_2) = c_V(a_1,\dots, a_1) + c_V(-a_2,\dots,-a_2)
\end{equation*}
 for any block $V \in \pi$, whence the equality follows. This new way of expressing $P_n(\epsilon, 0)$ hints to the even moments of the $1/2$-fold free convolution of the spectral distribution of $a_1-a_2$ (\cite{NS}).
Note that when $\theta= 1/2 \Leftrightarrow \epsilon = 0$ then $a_1,a_2, -a_2$ are distributed according to the symmetric Bernoulli distribution 
\begin{equation*}
\frac{1}{2}[\delta_1 + \delta_{-1}],
\end{equation*}
hence the $1/2$-fold free convolution of the spectral distribution of $a_1-a_2$ is still the symmetric Bernoulli distribution. Accordingly, we retrieve $P_n(0,0)$:
\begin{equation*}
P_n(0,0) = \int x^{2n} \frac{1}{2}[\delta_1+\delta_{-1}](dx) = 1.
\end{equation*}
However when $\theta \neq 1/2 \Leftrightarrow \epsilon \neq 0$ the situation becomes rather cumbersome: the spectral distribution of $a_1$ is given by 
\begin{equation*}
\theta \delta_1 + (1-\theta) \delta_{-1}
\end{equation*}
while that of $-a_2$ is given by
\begin{equation*}
(1-\theta)\theta \delta_1 + \theta \delta_{-1}.
\end{equation*}
Equivalently, the $R$-transform of $a_1$ reads 
\begin{equation*}
R_{a_1}(y) = \frac{\sqrt{1+4y(y+\kappa)} - 1}{2y}
\end{equation*}
while that of $-a_2$ reads 
\begin{equation*}
R_{-a_2}(y) = \frac{\sqrt{1+4y(y-\kappa)} - 1}{2y}
\end{equation*}
near $y=0$. It follows that the $R$-transform of the $1/2$-fold free convolution of $a_1-a_2$ is given by 
\begin{equation*}
R(y) := \frac{1}{2}[R_{a_1}(y) + R_{a_2}(y)] = \frac{\sqrt{1+4y(y+\kappa)} + \sqrt{1+4y(y-\kappa)}- 2}{4y}
\end{equation*}
and that its $K$-transform is given by 
\begin{equation*}
K(y) := R(y)+\frac{1}{y}  = \frac{\sqrt{1+4y(y+\kappa)} + \sqrt{1+4y(y-\kappa)}+2}{4y}.
\end{equation*}
Inverting $K$ (in composition sense) leads to the third-degree polynomial 
\begin{equation*}
y^3 - h_1(z)y^2 + h_2(z)y - h_3(z) = 0 
\end{equation*}
where 
\begin{eqnarray*}
h_1(z) &=& \frac{2z^2-1}{z(z^2-1)}\\
h_2(z) &=& \frac{5z^2 + \epsilon -1}{4z^2(z^2-1)} \\ 
h_3(z) &=& \frac{1}{4z(z^2-1)}.
\end{eqnarray*}

Performing the variables change 
\begin{equation*}
3y = [3h_2(z) - (h_1(z))^2]^{1/2}u + h_1(z)
\end{equation*}
we recover the reduced form 
\begin{equation*}
u^3 + 3u - \frac{2[h_1(z)]^3 - 9h_1(z)h_2(z) + 27h_3(z)}{[3h_2(z)- (h_1(z))^2]^{3/2}} = 0.
\end{equation*}
Solutions of this equation may be expressed through Gauss hypergeometric functions ${}_2F_1$ as shown in \cite{Hille} p.265-266.

\end{document}